\documentclass[12pt]{article}
\usepackage{amsmath,amssymb, amsthm, amscd, amsrefs}
\usepackage{pifont,mathrsfs}
\usepackage{array,lastpage}
\usepackage{enumerate,xspace,pifont,shadow}
\usepackage{mathrsfs}
\usepackage[T1]{fontenc}
\usepackage [english]{babel}
\usepackage [autostyle, english = american]{csquotes}

\MakeOuterQuote{"}

\DeclareSymbolFont{AMSb}{U}{msb}{m}{n}
\DeclareMathSymbol{\Z}{\mathbin}{AMSb}{"5A}
\DeclareMathSymbol{\R}{\mathbin}{AMSb}{"52}
\DeclareMathSymbol{\N}{\mathbin}{AMSb}{"4E}
\DeclareMathSymbol{\Q}{\mathbin}{AMSb}{"51}

\newcommand{\tp}{\textup{tp}}
\newcommand{\qftp}{\textup{qftp}}

\newcommand{\Th}{\textup{Th}}

\newcommand{\C}{\mathfrak{C}}
\newcommand{\floor}[1]{\lfloor #1 \rfloor}
\newcommand{\mc}[1]{\mathcal{#1}}
\newcommand{\mf}[1]{\mathfrak{#1}}
\newcommand{\ob}[1]{\overline{#1}}
\newcommand{\abs}[1]{\left \vert #1 \right \vert}

\def\Ind{\setbox0=\hbox{$x$}\kern\wd0\hbox to 0pt{\hss$\mid$\hss}
\lower.9\ht0\hbox to 0pt{\hss$\smile$\hss}\kern\wd0}

\def\Notind{\setbox0=\hbox{$x$}\kern\wd0\hbox to 0pt{\mathchardef
\nn=12854\hss$\nn$\kern1.4\wd0\hss}\hbox to
0pt{\hss$\mid$\hss}\lower.9\ht0 \hbox to
0pt{\hss$\smile$\hss}\kern\wd0}

\newtheorem{thm}{Theorem}[section]
\newtheorem{lem}[thm]{Lemma}
\newtheorem{cor}[thm]{Corollary}
\newtheorem{prop}[thm]{Proposition}

\newtheorem{fact}[thm]{Fact}

\newtheorem{quest}[thm]{Question}

\theoremstyle{definition}
\newtheorem{definition}[thm]{Definition}

\theoremstyle{remark}
\newtheorem{remark}[thm]{Remark}

\theoremstyle{remark}
\newtheorem{example}[thm]{Example}

\theoremstyle{remark}
\newtheorem{claim}[thm]{Claim}

\theoremstyle{remark}

\theoremstyle{remark}

\begin{document}
\bibliographystyle{amsra}

\title{Strong theories of ordered Abelian groups}

\author{Alfred Dolich\footnote{This work was partially supported by a grant from the Simon Foundation (\#240550 to Alfred Dolich)} \and John Goodrick}

\newcommand{\Addresses}{{
  \bigskip
  \footnotesize

  A.~Dolich, \textsc{Department of Mathematics and Computer Science Kingsborough Community College Brooklyn, NY \\  and \\ Department of Mathematics The CUNY Graduate Center NY,NY}\par\nopagebreak
  \textit{E-mail address}, A.~Dolich: \texttt{alfredo.dolich@kbcc.cuny.edu}

  \medskip

  J.~Goodrick, \textsc{Departamento de Matem\'aticas, Universidad de los Andes, Bogot\'a, Colombia}\par\nopagebreak
  \textit{E-mail}: \texttt{jr.goodrick427@uniandes.edu.co}

}}

\maketitle

\abstract{We consider strong expansions of the theory of ordered Abelian groups. We show that the assumption of strength has a multitude of desirable consequences for the structure of definable sets in such theories, in particular as relates to definable infinite discrete sets.  We also provide a range of examples of strong expansions of ordered Abelian groups which demonstrate the great variety of such theories.} 

\section{Introduction}

In this article we establish some tameness properties for discrete sets which are first-order definable in ordered Abelian groups, in languages extending $L_{oag} = \{0, +, <\}$, whose complete theories are strong. Below we will recall the definitions of ``strong,'' ``strongly dependent,'' and the other important properties of theories that we will consider and explain the motivation for their study.

Around 2005, Shelah proposed several possible definitions of what it should mean for a complete first-order theory to be ``strongly dependent,'' with corresponding ordinal-valued ranks, in an attempt to generalize the useful superstable / strictly stable dichotomy to the class of dependent (or NIP) theories, see \cite{strong_dep}. However, it should be emphasized that ``strongly dependent'' really generalizes the idea that types have bounded weight rather than superstability, and there are stable, non-superstable theories which are strongly dependent. Then Adler, in \cite{adler_strong_dep}, gave the formulations of \emph{strongly dependent} and \emph{inp-pattern} which are now standard (and which we follow), as well as introducing the more general class of strong theories (generalizing ``strongly dependent'').

The logical relations among some of the properties of theories relevant to this paper are summarized below, where ``$X \Rightarrow Y$'' signifies that any theory which is $X$ is also $Y$, and all implications are non-reversible.

\bigskip

$\begin{array}{ccccccc}
 \textup{inp-minimal} & \Rightarrow & \textup{finite burden} & \Rightarrow & \textup{strong} & \Rightarrow & \textup{NTP}_2 \\
\Uparrow & & \Uparrow &  & \Uparrow & & \Uparrow \\
 \textup{dp-minimal} & \Rightarrow & \textup{finite dp-rank} & \Rightarrow & \textup{strongly NIP} & \Rightarrow & \textup{NIP} \\

\end{array}$ \\

There is a significant body of previous and concurrent work on algebraic structures (groups, rings, fields, valued fields, \emph{et cetera}) which have theories satisfying the various properties mentioned above. For example, a motivating conjecture is that any NIP field is either algebraically closed, real closed, or admits a definable Henselian valuation; this is still open even for the class of stable fields, although something very close to this was proved very recently by Will Johnson in \cite{Johnson}. Similar results were independently obtained by Jahnke, Simon, and Walsberg in \cite{JSW}, who also characterized dp-minimal ordered Abelian groups and dp-minimal ordered fields in the pure languages of ordered groups and rings.

Other related recent work includes the study of dp- and inp-minimal ordered structures in \cite{Goodrick_dpmin} and \cite{S} and the existence of Abelian, solvable and nilpotent definable envelopes of groups definable in $\textup{NTP}_2$ theories by Hempel and Onshuus \cite{Hempel_Onshuus}. The new work by Simon and Walsberg on dp-minimal structures with tame topology in \cite{Simon_Walsberg} also explores similar ideas, and it would be interesting to see how some of their results might be generalizable to the strong or finite dp-rank context.

We obtain the following general results:

\begin{enumerate}
\item Fields with finite burden satisfy uniform finiteness (Corollary~\ref{UF});
\item In an ordered Abelian group $\langle R; +, <, \ldots\rangle$ with a strong theory, no definable discrete $X \subseteq R$ has accumulation points (Corollary~\ref{nowheredense} (1));
\item In an ordered field $\langle R; <, +, \cdots \rangle$ with a strong theory, there are no infinite definable discrete sets $X \subseteq R$ (Corollary~\ref{nowheredense} (2));
\item In a definably complete ordered field $\langle R; <, +, \cdots\rangle$ with a strong theory, any nowhere dense definable $X \subseteq R$ is finite (Corollary~\ref{nowheredense}, (2) and (3));
\item In a densely-ordered, strong, definably complete Abelian group $\langle R; <, +, \cdots \rangle$, the image of a discrete set under a definable function is always nowhere dense (Corollary~\ref{images_discrete});
\item In a strong Archimedean ordered Abelian group $\langle R; <, +, \ldots \rangle$, any discrete definable subset $P \subseteq R$ must be a finite union of arithmetic sequences (Theorem~\ref{discrete_R}).
\end{enumerate}

We also give (in Section~3 below) a series of new examples illustrating the richness of the class of ordered Abelian groups with strong theories. Here are some of the main motivating examples:


\begin{enumerate}
\item Any $o$-minimal group (which is necessarily Abelian and divisible) is dp-minimal, hence strongly dependent, as is any weakly $o$-minimal group (see \cite{DGL}).
\item An expansion of an $o$-minimal group by a generic predicate is inp-minimal, hence strong, but the theory has the independence property (see Proposition~\ref{genexpansion} below). Thus dense, codense definable sets are possible in the context of this paper.
\item The complete theory of the ordered additive group $\langle \R; +, < \rangle$ expanded with predicates for both $\Z$ and $\Q$ has finite dp-rank (see Proposition~\ref{zq} below), hence is strongly dependent, so proper discrete subgroups may be definable.
\item There are finite dp-rank (hence strongly dependent) expansions of $\langle \R; +, < \rangle$ with definable functions $f: \R \rightarrow \R$ whose graphs are dense in $\R^2$, see Proposition~\ref{dense_graphs} below.
\item Additive reducts of tame pairs of o-minimal fields have finite dp-rank -- see Proposition~\ref{tame_pair} below. This provides examples in which there are definable discrete subsets $P$ of the universe but no element of $P$ has an immediate successor or predecessor.
\item The complete theory of $\langle \Z; +, < \rangle$ (that is, Presburger arithmetic) is dp-minimal. However, if we add a unary predicate $P$ for the set $\{2^k : k \in \N\}$, then the complete theory of the structure is no longer strong. In fact, our results below imply that the only case in which an expansion of this structure by a unary predicate $P$ is strong is when $P$ is already Presburger definable.
\end{enumerate}

The authors would like to thank the organizers of the 2015 meeting on Neostability Theory in Banff International Research Station - Oaxaca for a very stimulating event at which we had many productive discussions on the topics of this paper, and Anand Pillay for the observation in the proof of Corollary~\ref{strong_exp}.

\subsection{Basic definitions and notation}

``Definable'' will include sets that are definable over parameters, unless otherwise specified. We always work with a complete first-order, $1$-sorted theory $T$ and a large ``monster'' model $\C \models T$ from which we pick all of our parameters, where $\C$ is assumed to be $\kappa$-saturated for some $\kappa$ larger than $|T|^+$ and any set of parameters that we happen to be working over.

The following definitions (due to Adler \cite{adler_strong_dep}) are fundamental to all that follows.

\begin{definition}
\label{inp}
For any cardinal $\kappa$, an \emph{inp-pattern of depth $\kappa$} (``independent partition'') for a partial type $p(\overline{x})$ is a sequence of pairs $\langle (\varphi_\alpha ( \overline{x} ; \overline{y}_\alpha), k_\alpha) : \alpha < \kappa \rangle$ where $\varphi_\alpha$ is a formula, $k_\alpha$ is a positive integer, and there is an array of tuples $\langle \overline{b}^{\alpha}_i : \alpha < \kappa, i < \omega \rangle$ witnessing it: that is, for each $\alpha < \kappa$, the row $\{\varphi_\alpha(\overline{x}; \overline{b}^\alpha_i) : i < \omega \}$ is $k_\alpha$-inconsistent, while for every function $\eta : \kappa \rightarrow \omega$, the path $\{\varphi_\alpha(\overline{x}; \overline{b}^\alpha_{\eta(\alpha)}) : \alpha < \kappa\}$ is consistent with $p(\overline{x})$.
\end{definition}

\begin{definition}
\label{ict}
For any cardinal $\kappa$, an \emph{ict-pattern of depth $\kappa$} (``independent contradictory types'') for a partial type $p(\overline{x})$ is a sequence of formulas $\langle \varphi_\alpha ( \overline{x} ; \overline{y}_\alpha) : \alpha < \kappa \rangle$ such that there is an array of tuples $\langle \overline{b}^{\alpha}_i : \alpha < \kappa, i < \omega \rangle$ witnessing it: that is, for every function $\eta : \kappa \rightarrow \omega$, the partial type $$\{\varphi_\alpha(\overline{x}; \overline{b}^\alpha_{\eta(\alpha)}) : \alpha < \kappa\} \cup \{\neg \varphi_\alpha(\overline{x} ; \overline{b}^\alpha_i) : i \neq \eta(\alpha) \}$$ is consistent with $p(\overline{x})$.
\end{definition}

The maximal cardinal $\kappa$ (if it exists) such that there is an inp-pattern of depth $\kappa$ for $p(\overline{x})$ can be thought of as an analogue of the concept of ``weight'' in stable theories (in \cite{adler_strong_dep} it is called the \emph{burden} of $p$). The arrays of parameters witnessing inp- and ict-patterns can always be chosen so that each row is indiscernible over the union of the remaining rows. See \cite{adler_strong_dep} for this and other basic facts.

\begin{definition}
\begin{enumerate}
\item The \emph{burden} (or \emph{inp-rank}) of a theory $T$ is the minimal cardinal $\kappa$ (if it exists) such that there is no inp-pattern in a single free variable $x$ of depth $\kappa^+$. A theory has \emph{finite burden} if it has burden $\leq n$ for some $n \in \N$, and it is \emph{inp-minimal} if its burden is $1$. We say that a structure has finite burden (is inp-minimal, \emph{et cetera}) if its complete theory does.
\item The \emph{dp-rank} of a theory $T$ is the minimal cardinal $\kappa$ (if it exists) such that there is no ict-pattern in a single free variable $x$ of depth $\kappa^+$. A theory has \emph{finite dp-rank} if it has dp-rank $\leq n$ for some $n \in \N$, and it is \emph{dp-minimal} if its dp-rank is $1$.
\item A theory is \emph{strongly dependent} if every ict-pattern in finitely many variables is finite.
\item A theory is \emph{strong} if every inp-pattern in finitely many variables is finite.
\end{enumerate}
\end{definition}

Note that ``strongly dependent'' is equivalent to ``strong and NIP,'' and every strong theory is $\textup{NTP}_2$. See \cite{additivity_dp_rank} for more basic results, such as the sub-additivity of dp-rank.  Also note that our definition of inp-rank and dp-rank differ slightly from those in \cite{adler_strong_dep}.  This difference only manifests itself in the situations where the inp or dp rank is infinite and as in this paper we are concerned with the finite rank cases the difference is immaterial.


\section{General results}

\subsection{Finite inp-rank fields}

We begin with some basic results on fields of finite inp-rank, in particular in this subsection we do not assume that the structures in which we are working are ordered.  Let  $\mathfrak{F} =  \langle F; +, \cdot, \ldots \rangle$ be a field $F$ whose theory has finite inp-rank. Notice that we allow for the situation where $\mathfrak{F}$ is a structure in a language richer than the usual language of rings.

\begin{lem}
\label{large_def_sets}
If $X \subseteq F$ is definable and infinite and $\langle X \rangle$ is the subfield of $F$ generated by $X$, then $F$ is a finite algebraic extension of $\langle X \rangle$.
\end{lem}

\begin{proof}
Suppose not. Then for any $n < \omega$, we can find a collection of elements $\{a_1, \ldots, a_n\}$ from $F$ which are $\langle X \rangle$-linearly independent. Then we can define a function $g: X^n \rightarrow F$ by $g(x_1, \ldots, x_n) = \sum_{i=1}^n x_i a_i$ and by linear independence, $g$ is injective. The existence of such functions $g$ implies that the inp-rank of $F$ is at least $n$ for any $n$, contradiction.
\end{proof}

This leads to the next corollary, which was proved in the special case of dp-minimal fields by Johnson in \cite{Johnson}:

\begin{cor}
\label{UF}
The structure $\mathfrak{F}$ has uniform finiteness (UF): for any formula $\varphi(x;\overline{y})$, there is a number $n < \omega$ such that whenever $|\varphi(F; \overline{b})| > n$, then $\varphi(F; \overline{b})$ is infinite.
\end{cor}

\begin{proof}
If $X_a := \varphi(F; \overline{b})$ is finite, then clearly there is a collection $\{a_i : i < \omega\}$ of elements of $F$ which are linearly independent over the field generated by $X_a$. If UF failed, then we could use compactness to find an $a$ such that $X_a$ is infinite and a collection of $\{a_i : i < \omega\}$ linearly independent over the field generated by $X_a$, contradicting Lemma~\ref{large_def_sets}.
\end{proof}

We can consider the special case where $\mathfrak{F}$ is an ordered field. 
Recall the following basic definition (for a discussion see \cite{ivp}):

\begin{definition} A structure $ \langle R; <, \dots \rangle$ modeling the theory of dense linear orderings is called {\em definably complete} if for any definable $X \subseteq R$ which is bounded above (below) has a least upper bound (greatest lower bound).
\end{definition}

When $\mathfrak{F}$ is ordered Corollary \ref{UF} has a strong consequence (for o-minimal open core see \cite{opencore2}):

\begin{cor} if $\mathfrak{F}$ is ordered and  definably complete then any model of $\Th(\mathfrak{F})$ has o-minimal open core.
\end{cor}

\begin{quest}
Does UF hold for strongly dependent fields?
\end{quest}

\begin{quest}
Is there a finite dp-rank ordered field with a definable dense and codense subset?  In other words is there a theory $T$ all of whose models have o-minimal open core and such that $T$ has finite dp-rank? 
\end{quest}

By way of contrast, note that an expansion of a real-closed field by a generic unary predicate (as discussed in Section~3.5 below) is an inp-minimal ordered field with a dense codense predicate.

\subsection{Ordered Abelian groups}

\textbf{Here and below, $\mathcal{R} =  \langle R;  <, +, \ldots \rangle$ always denotes an ordered Abelian group.} We consider $R$ as a topological space with the order topology and $R^n$ as a topological space with the corresponding product topology. Note that some of the topological results below will be trivial in the case of a discretely-ordered group such as $\mathcal{R} = \langle \Z; <, + \rangle$, but everything holds for any strong ordered Abelian group, unless we specifically say that it is dense.

For $x \in R$, we let $|x| = \max(x, -x)$ and define the ``distance'' $d(x,y)$ between two elements $x$ and $y$ as $|x - y|$. Note that this gives an $R$-valued metric on $R$ satisfying the usual triangle inequality.

\begin{prop}
\label{indep}
Suppose that $S \subseteq R$ is definable and infinite, $n, k \in \N$, and for every $g \in R$ there are at most $k$ distinct ways to express $g$ as a sum $s_1 + \ldots + s_n$ where each $s_i \in S$. Then the burden of $x=x$ is at least $n$.
\end{prop}

\begin{proof}
Let $\{a_{i,j} : i \in \omega, 1 \leq j \leq n\}$ be any collection of pairwise distinct elements from $S$, and let $\varphi(x, a)$ be a formula expressing ``$x-a$ is a sum of $(n-1)$ elements from $S$.'' Then the array $\{\varphi(x; a_{i,j}) : i \in \omega, 1 \leq j \leq n\}$ witnesses that $x=x$ has burden at least $n$.
\end{proof}

\begin{example}
Any ordered Abelian group with a $\Q$-independent definable set $S \subseteq R$ does not have finite burden.
\end{example}

\begin{example}
Consider the compete theory of $\langle \R; <, +, P \rangle$ where $P = \{2^k : k \in \N\}$. Then $P$ is not $\Q$-independent, but it satisfies the hypotheses of $S$ in Proposition~\ref{indep} with $k=1$ (by the uniqueness of representations with binary digits). Thus the theory does not have finite burden.
\end{example}

Our next Lemma is a fundamental technical statement we will use to derive several subsequent results.

\begin{lem}
\label{discrete_family}
Suppose that there is a family of infinite definable discrete sets $D_i$ and $\varepsilon_i>0$ for $i \in \N$ such that:

\begin{enumerate}

\item $3 \cdot D_{i+1} \subseteq (0, \varepsilon_i)$, where $3 \cdot D_{i+1}$ refers to the set $\{3 x : x \in D_{i+1}\}$; and
\item If $x \in D_i$ then $(x-\varepsilon_i,x+\varepsilon_i) \cap D_i = \{x\}$.

\end{enumerate}

Then $\mathcal{R}$ is not strong.

\end{lem}

\begin{proof}
Let $Y_0=D_0$ and inductively define sets $Y_n$ for $n>0$ as follows:
\[Y_{n+1}=\{x : x = y+d \text{ so that } y \in Y_n \text{ and } d \in D_{n+1}\}.\]

For each $j \in \N$ pick pairwise disjoint open intervals $I^j_i$ for $i \in \N$ so that $I^j_i \cap D_j$ is infinite for all $i$.
Let $l^j_i$ and $r^j_i$ be the left and right endpoints of the $I^j_i$ respectively.

We claim that if $a,b, \in Y_n$ are distinct then $\abs{a-b}>\varepsilon_n$.  We verify this by induction on $n$.  If $n=0$ this is trivial by construction.  Now suppose that $a,b \in Y_{n+1}$.  By definition $a=y+d$ and $b=z+e$ where $y,z \in Y_n$ and 
$d,e \in D_{n+1}$.  Suppose that $y \not = z$ By induction $3 \abs{y-z}> 3 \varepsilon_n$ and by construction $3 \abs{d}<\varepsilon_n$ and $3 \abs{e}<\varepsilon_n$ so $$3 \cdot \abs{a-b} = \abs{3 (y - z) + (3d - 3e)} > \varepsilon_n >3  \varepsilon_{n+1}.$$  Hence $y=z$ but then 
$\abs{a-b}=\abs{d-e}>\varepsilon_{n+1}$ by construction.

Now let $\varphi_j(x;y,z)$ for $j \geq 1$ be the formula: \[\exists w (w \in Y_{j-1} \wedge w<x \wedge y<x-w<z)\] and consider the formulae $\varphi_j(x;l^j_i,r^j_i)$ for $i \in  \N$.  We claim that $\varphi_j(x;l^j_i,r^j_i)\wedge \varphi_j(x;l^j_{k},r^j_{k})$ is inconsistent if $i \not= k$.  Otherwise suppose that $a$ realizes the conjunction.  We find $y_i, y_k  \in Y_{j-1}$ so that:  $y_i<a$, $a-y_i \in I^j_i$, $y_k<a$, and $a-y_k \in I^j_k$.  First notice that we must have that $y_i=y_k$, else suppose that $y_i<y_k$, but then $a-y_i >\varepsilon_{j-1}$ hence $a-y_i \notin I_i^j$ since 
$3x \in (0, \varepsilon_{j-1})$ for every $x \in I_i^j$. Thus $y_i=y_k$ but then $a-y_i$ must lie in both $I^j_i$ and $I^j_k$, a contradiction. 

Now let $\eta : \N \setminus \{0\} \to \N$.  We claim that the type \[\Gamma(x)=\bigwedge_{j \in \N}\varphi_j(x; \l^j_{\eta(j)},r^j_{\eta(j)})\] is consistent.  Let $y_0 \in Y_0$ and let $J_0$ be the open interval 
$(y_0+l^1_{\eta(1)},y_0+r^1_{\eta(1)})$.   Choose $d \in D_1 \cap I^1_{\eta(1)}$ so that $d$ is neither the smallest or largest element in this intersection (recall that the intersection is assumed to be infinite) and set $y_1=y_0+d$.  
Hence $y_1 \in Y_1 \cap J_0$. Now let $J_1=(y_1+l^2_{\eta(2)},y_1+r^2_{\eta(2)})$.  We claim that $J_1 \subseteq J_0$.  As $d$ was neither the smallest or largest element in $D_1 \cap I^1_{\eta(1)}$  we can find $d_0<d$ and $d_1>d$ with $d_i \in D_1$ so that $y_0+d_0$ and $y_0+d_1$ both lie in $Y_1 \cap J_0$.  Hence the interval $(y_0+d_0, y_0+d_1)$ is contained in $J_0$.  But both $3 l^2_{\eta(2)}$ and $3 r^2_{\eta(2)}$ are less than $\varepsilon_1$ and so $J_1 \subseteq (y_0+d_0,y_0+d_1) \subseteq J_0$. as desired.   Now choose $d \in D_2 \cap I^2_{\eta(2)}$ and set $y_2=y_1+d$.  Continuing in this manner we construct a sequence of open intervals $J_i$ for $i \in \N$ so that $J_{i+1} \subseteq J_i$.  Finally notice that if $a \in \bigcap_{i \in \N}J_i$ then $a$ realizes $\Gamma(x)$.

Thus we have constructed an inp-pattern of depth $\aleph_0$ and so the theory is not strong.

\end{proof}

Our previous Lemma allows us to derive the following theorem which severely constrains the type of discrete sets that are definable in a strong structure $\mc{R}$.

\begin{thm}\label{small_discrete} Suppose that $\mc{R}= \langle R; +, <, \dots \rangle$ is a densely-ordered Abelian group.  Let $\C \models Th(\mc{R})$ be a saturated model and suppose that for every $\varepsilon > 0$ there is an infinite definable discrete set $X$ such that $X \subseteq (0,\varepsilon)$. Then $Th(\mc{R})$ is not strong.
\end{thm}

\begin{proof}

Given $\mc{R}$ as in the hypothesis, we simply have to construct a family of definable sets $D_i$ as in Lemma~\ref{discrete_family}. We do this by induction on $\N$.  Let $D$ be any definable infinite discrete set and without loss of generality assume that $D \subseteq (0,\infty)$.  Let $\varepsilon>0$ be such that there are infinitely many $x \in D$ so that $(x-\varepsilon,x+\varepsilon) \cap D=\{x\}$.  Let $D_0=\{x \in D : (x-\varepsilon,x+\varepsilon) \cap D = \{x\}\}$.  Let $\varepsilon_0=\varepsilon$.

Suppose we have constructed $D_n$ and $\varepsilon_n$.  Let $\varepsilon' \in R$ be such that $0 < 3 \varepsilon' < \varepsilon_n$ and let $D$ be an infinite definable discrete set so that $D \subseteq (0, \varepsilon')$.  Pick $\varepsilon \in (0, \varepsilon')$ so that there are infinitely many $x \in D$ such that $(x-\varepsilon,x+\varepsilon) \cap D =\{x\}$.  
Let $D_{n+1}=\{x \in D : (x-\varepsilon,x+\varepsilon) \cap D =\{x\}\}$ and $\varepsilon_{n+1}=\varepsilon$.

\end{proof}

Next we will prove some topological tameness properties of unary sets definable in strong densely-ordered Abelian groups.\footnote{Technically speaking these results are true in any strong ordered Abelian group, but when the ordering is discrete they are completely trivial since every point is isolated.}

If $S \subseteq R$, define $S'$ to be the usual Cantor-Bendixson derivative, that is, $$S' = \{s \in R :  \textup{ whenever } a < s < b, \, \left|(a, b) \cap S \right|>1\}.$$ Note that if $S$ is not closed then $S'$ is not a subset of $S$. Define the set of \emph{limit points of $S$ from the left} to be $$S'_L = \{x \in R : \textup{ for every } a < x, \, (a, x) \cap S \neq \emptyset\},$$ and we similarly define $S'_R$, the set of \emph{limit points of $S$ from the right}. A point $x$ is \emph{left-isolated} (or \emph{right-isolated}) if $x \in S \setminus S'_L$ (respectively, $x \in S \setminus S'_R$).

\begin{cor}
\label{limitpts}
Suppose that $\mc{R}=\langle R; +, <, \dots \rangle$ is a densely-ordered Abelian group with a strong theory and $S \subseteq R$ is definable.

\begin{enumerate}
\item No $x \in R$ is a limit point of $S \setminus S'$. 

\item No $x \in R$ is a limit point of $S \setminus S'_R$ from the right.

\item No $x \in R$ is a limit point of $S \setminus S'_L$ from the left.

\end{enumerate}

\end{cor}

\begin{proof}
We only write the proof of (2), since (1) follows by a similar proof, and (3) follows from (2) applied to $-S$. Suppose, the contrary, $x \in R$ is a limit point of $S \setminus S'_R$ from the right. Without loss of generality, $x = 0$ (moving $S$ by a translation). Now for any $\epsilon > 0$, there are infinitely many points of $S \setminus S'_R$ in $(0, \epsilon)$, and so by compactness there must be some $\delta \in R$, $\delta > 0$ such that $$\{a \in S : (a, a + \delta) = \emptyset \} \cap (0, \epsilon)$$ is infinite. But the set above is discrete, so we have contradicted Theorem~\ref{small_discrete}.
\end{proof}

In particular notice that the above corollary implies that in a structure $\mc{R}$ with strong theory for no definable set $S$ can $S^{\prime}$ be finite and nonempty, since otherwise each point in $S^{\prime}$ would be a limit point of $S \setminus S^{\prime}$.

We summarize several further corollaries of our previous results.

\begin{cor}\label{nowheredense} Let $\mc{R}=\langle R; +, <, \dots \rangle$ be a densely-ordered Abelian group and suppose that $\Th(\mc{R})$ is strong.  Let $\C$ be a big model of $\Th(\mc{R})$.  Then:
\begin{enumerate}
\item If $X$ is an infinite definable discrete set in $\C$ then $X$ has no accumulation point in $\C$.  In particular $X$ must be closed.
\item If $\mc{R}$ is an expansion of an ordered field then in $\C$ there are no infinite definable discrete sets.
\item If $\mc{R}$ is definably complete and  $X \subseteq \C$ is definable and nowhere dense then $X$ is discrete.
\end{enumerate}
\end{cor} 

\begin{proof}  (1) is an immediate consequence of Lemma \ref{discrete_family}.  For (2) let $\C$ have field structure and suppose that $D \subseteq C$ is discrete and definable.  Without loss of generality assume that $D \subseteq C^{>0}$.  If $D$ is unbounded then the 
set $D^{-1}=\{x^{-1}: x \in D\}$ is discrete and has $0$ has an accumulation point violating (1).  Hence $D$ is bounded.  Say that $\Delta>d$ for all $d \in D$.  
But then for any $\varepsilon>0$ the set $D_{\varepsilon}:=\{\frac{\varepsilon}{\Delta}d: d \in D\}$ is a discrete subset of $(0, \varepsilon)$.  Hence by Theorem \ref{small_discrete} $D$ must be finite.  For (3) suppose that $X \subset C$ is infinite, nowhere dense, and not discrete.  Without loss of generality $X$ is closed.  Thus $C\setminus X$ is open and so by definable completeness consists of an infinite family of open intervals.  Let $Y$ be the set of all midpoints of intervals in $C\setminus X$, which we easily see is definable.  But then $Y$ is discrete and infinite, moreover as $X$ is not discrete then $Y$ must have an accumulation point, violating (1).
\end{proof}

The following corollary provides a counterpoint to Theorem \ref{small_discrete}, namely we can not have definable discrete sets which are too "spread out".

\begin{cor}  
\label{discr_sparse}
Suppose that $\mc{R}= \langle R; +, <, \dots \rangle$ is any ordered Abelian group and that $T=\Th(\mc{R})$ is strong.  Let $D \subseteq R$ be an infinite definable discrete set which is not bounded above.  For any $\delta > 0$, let $$D(\delta):=\{ x \in D : (x-\delta, x+\delta) \cap D =\{x\}\}.$$ Then there is some positive $\delta \in R$ such that $D(\delta) = \emptyset$.
\end{cor} 

\begin{proof} Otherwise for any $\delta \in R$ the set $D(\delta)$ is infinite.  We will show that, in some nonstandard model, there is a family of sets of the form $D_i = D(\varepsilon_i) \cap I_i$ which satisfy the hypotheses of Lemma~\ref{discrete_family}, where $I_i$ is the open interval $(\ell_i, r_i)$.

We do this by choosing all of the parameters $\ell_i, r_i$, and $\varepsilon_i$ using compactness. It is enough to show that the following partial type $\Gamma$ is consistent:

\begin{enumerate}[i]
\item $(3 \ell_{i+1}, 3 r_{i+1}) \subseteq (0, \varepsilon_i)$;
\item $0 < 3 \varepsilon_{i+1} < \varepsilon_i$; and
\item For every $n \in \N$, $| D(\varepsilon_i) \cap (\ell_i, r_i) | \geq n$.
\end{enumerate}

We claim that any finite $\Gamma_0 \subseteq \Gamma$ is satisfied by values of $\ell_i, r_i,$ and $\varepsilon_i$ in the model $\mc{R}$ and conclude by compactness. To see this, first pick a sufficiently large $N \in \N$ such that all the $\ell_i, r_i,$ and $\varepsilon_i$ mentioned in $\Gamma_0$ satisfy $i \in \{1, \ldots, N\}$ and also $N$ bounds any $n$ from a formula of type (iii) in $\Gamma_0$. Pick $\varepsilon_N > 0$ arbitrarily, then pick $\ell_N, r_N$ positive and sufficiently far apart that $| D(\varepsilon_N) \cap (\ell_N, r_N) | \geq N$. Then continue by reverse induction on $i$: once we have picked $\ell_{i+1}, r_{i+1},$ and $\varepsilon_{i+1}$ (for $i \geq 0$), first pick $\varepsilon_i > \max(3 \varepsilon_{i+1}, 3 r_{i+1})$, then pick $\ell_i$ and $r_i$ far enough apart that $|D(\varepsilon_i) \cap (\ell_i, r_i)| \geq N$.

\end{proof}

\begin{cor} The theory of the structure $\langle \R; +, <, 2^{\Z} \rangle$ is not strong, and the theory of the structure $\langle \Z; +, <, 2^{\N} \rangle$ is not strong.
\end{cor}

The following Proposition can be thought of as a version of the definable Baire property (for which see \cite{forservi}).

\begin{prop}\label{baire}  Suppose that $\mc{R}= \langle R; +, <, \dots \rangle$ is a densely ordered Abelian group such that $\Th(\mc{R})$ is strong.  We may not find an interval $I$ and a definable family $D(\delta)$ for $\delta \in I$ so that each $D(\delta)$ is discrete, if $\delta_1<\delta_2 \in I$ then $D(\delta_1) \subseteq D(\delta_2)$, and $D=\bigcup_{\delta \in I}D(\delta)$ is somewhere dense.
\end{prop}

\begin{proof}  This follows readily from Proposition \ref{discrete_family}.  Suppose that there were a counterexample $D(\delta)$ with $\delta \in I$ so that $D$ is dense in an interval $J$.  We build a family $D_n,\varepsilon_n$ violating Proposition \ref{discrete_family}.  By compactness choose $\delta_0 \in I$ so that $D(\delta_0) \cap J$ is infinite.  Pick $\varepsilon_0$ so that 
\[D_0:=\{x \in D(\delta_0) \cap J : (x-\varepsilon_0,x+\varepsilon_0) \cap D(\delta) = \{x\}\}\] is infinite.  
Thus we have $D_0$ and $\varepsilon_0$.  Next pick $a \in D_0 \cap J$ so that $(a, a+\varepsilon_0) \subseteq J$.  As $D$ is dense in $J$ we may find $\delta_1 \in I$ so that $3 \cdot D(\delta_1) \cap (a, a+\varepsilon_0)$ is infinite.  Pick $\varepsilon_1$ so that 
\[D^{\prime}:=\{x \in 3 \cdot D(\delta_1) \cap (a, a+\varepsilon_0): (x-\varepsilon_1, x+\varepsilon_1) \cap D(\delta_1)=\{x\}\}\] is infinite.
Let $D_1=\{x-a: x \in D^{\prime}\}$.  Continue in this manner to construct a sequence $D_n, \varepsilon_n$ for $n \in \omega$ violating Proposition \ref{discrete_family}.
\end{proof}

This  form of the Baire property has the following corollary, indicating that the image of a discrete set under a definable function must be nowhere dense.  This should be compared  to results from, for example, \cite{phillipdisc} where having a discrete set whose image under a definable function whose image is somewhere dense strongly indicates the structure in question is "wild".

\begin{cor}  
\label{images_discrete}
Let $\mc{R}= \langle R; +, <, \dots \rangle$.  Suppose that $\Th(\mc{R})$ is strong, densely-ordered, and satisfies DC.  If $P$ is a definable discrete set and $f: R^n \to R$ is a definable function then $f[P^n]$ is discrete.
\end{cor}

\begin{proof}  We begin by noting that by Corollary \ref{nowheredense}(3) it suffices to show that $f[P^n]$ is nowhere dense. Suppose this fails.  By Corollary \ref{nowheredense}(1) $P$ is closed.   Without loss of generality we may  assume that $f[P^n] \subseteq I$ for an interval $I$ and that $f[P^n]$ is dense in $I$.  Finally for convenience we may assume that $0 \in P$.

We proceed by induction on $n$.  Suppose that $n=1$.  For $\delta \in R^{>0}$ let 
$P(\delta)=P \cap (-\delta,\delta)$.  Thus $f[P]=\bigcup_{\delta \in R^{>0}}f[P(\delta)]$.  As $P$ is closed and discrete with least element $0$ and $\Th(\mathcal{R})$ is definably complete it must be the case that $P(\delta)$ is finite and non-empty for small enough $\delta$.  In particular $f[P(\delta)]$ is discrete for all small enough $\delta$.  By Proposition \ref{baire} and Corollary \ref{nowheredense} $f[P(\delta)]$ must be dense in $I$ for all sufficiently large $\delta$.  Thus by definable completeness we may find $\Delta \in R$, the supremum of all $\delta \in R$ so that $f[P(\delta)]$ is discrete.  Note that as $P$ is closed $\Delta \in P$ or $-\Delta \in P$.  Hence $f[P(\Delta)]$ is discrete but $f[P(\Delta)] \cup \{f(\Delta)\}\cup\{f(-\Delta)\}$ is somewhere dense in $I$, which is clearly absurd.

Now suppose we have $f[P^{n+1}]$ dense in $I$.  For $\delta \in R^{>0}$ let 
$P(\delta)=P \cap (-\delta, \delta)^{n+1}$.  As in the previous case we find $\Delta \in R^{>0}$ which is the supremum of all $\delta$ so that $f[P(\delta)]$ is discrete and note that $\Delta \in P$ or $-\Delta \in P$. 
 For $1 \leq i \leq n+1$ let $P^+_i$ be the set \[\{ \ob{a} \in P^{n+1} : -\Delta \leq a_j \leq \Delta \text{ for all } 1 \leq j \leq n+1 \text{ and } a_i = \Delta\}\] and let $P_i^-$ be defined analogously with $\Delta$ replaced by $-\Delta$.
Thus $f[P(\Delta)^{n+1}]$ is discrete but
 \[f[P(\Delta)^{n+1}] \cup \bigcup_{1 \leq i \leq n+1}f[P^+_i] \cup \bigcup_{1 \leq i \leq n+1}f[P^-_i]\] is somewhere dense in $I$.  Thus $f[P^{\circ}_{i^*}]$ is somewhere dense for some $i^*$ and $\circ \in \{+,-\}$, for simplicity assume that $i^*=1$ and $\circ=+$.  But then setting $\tilde{f}: R^n \to R$ to be the function $\tilde{f}(x_1, \dots x_n)=f(\Delta,x_1, \dots x_n)$ we have that $\tilde{f}[P^n]$ is somewhere dense and we are done by induction.

\end{proof}

\subsection{Discrete definable sets in Archimedean ordered groups}

We continue the study of definable discrete unary sets in strong ordered Abelian groups $\mc{R} = \langle R; +, <, \ldots \rangle$, focusing on the case where there is a model $R$ of the theory which is \emph{Archimedean}: that is, for any two $a, b \in R^{>0}$, there is some $n \in \N$ such that $n a > b$ and $n b > a$. The goal will be the following theorem, which will be proved by a series of lemmas:

\begin{thm}
\label{discrete_R}
Suppose that $\mc{R}= \langle R; +, <, \dots \rangle$ is an ordered Abelian group whose complete theory is strong and such that the universe $R$ is Archimedean, and let $P \subseteq R$ be definable, infinite, and discrete. Then:

\begin{enumerate}
\item The set $\Delta P$ of all differences $x - y$ such that $x, y \in P$ and $x$ is the immediate successor of $y$ in $P$, is finite;
\item $P = F \cup X_1 \cup \ldots \cup X_k$ where $F$ is finite and each $X_i$ is an infinite $\N$-indexed arithmetic progression (i.e. a set of the form $\{a + b \cdot i : i \in \N\}$ for some fixed $a,b \in R$);
\item The $X_i$ above may be taken to be ``commensurable'' in the sense that there exists a single $\Z$-indexed arithmetic progression $Y$ containing each $X_i$ as a subsequence; and
\item If $f: P \rightarrow R$ is definable, then the image of $f$ is discrete.
\end{enumerate}
\end{thm}

\begin{remark}
Part (1) is an easy corollary of part (2), but we give it special mention because it will still be true even in ``nonstandard'' models of $\Th(\mc{R})$ in which the universe is not Archimedean, whereas parts (2) and (3) which mention arithmetic sequences do not have clear analogues for non-Archimedean structures. Part (4) is also true in nonstandard models.

Note that some of the progressions $X_i$ in part (2) may tend to $-\infty$ while others tend to $\infty$, which is why in part (3) we require a progression of the form $Y = \{a + b \cdot i : i \in \Z\}$.
\end{remark}

In the special case when $R = \Z$, Theorem~\ref{discrete_R} is reminiscent of Szemer\'edi's Theorem from combinatorics (see \cite{Sz}). Indeed, if $P \subseteq \Z$ is infinite and definable in a strong theory, then Corollary~\ref{discr_sparse} above implies that $P$ has positive upper density, and so for every $k \in \N$, the set $P$ must contain an arithmetic sequence of length $k$. However, conclusion (2) of Theorem~\ref{discrete_R} is much stronger than this (and derived from stronger hypotheses), and our proof is independent of Szemer\'edi's result.

We also note in passing the following generalization of Proposition~6.6 of \cite{ADHMS}, which states that no proper expansion of $\langle \Z; <, + \rangle$ is dp-minimal.

\begin{cor}
\label{strong_exp}
If $m \in \N$, $R \subseteq \Z^m$, and $\Th(\Z,; <, +, R)$ is strong, then $R$ is definable in $\Th(\Z; <, +)$.
\end{cor}

\begin{proof}
By Theorem~5.1 of \cite{mv}, if $R \subseteq \Z^m$ is not definable in $\langle \Z; <, +\rangle$, then there is some $R' \subseteq \Z$ which is definable in $\langle \Z; <, +, R \rangle$ but not in $\langle \Z; <, +\rangle$. Now apply Theorem~\ref{discrete_R}.
\end{proof}

\begin{quest}
Suppose that we omit the hypotheses in Theorem~\ref{discrete_R} that $R$ is Archimedean, but add the hypothesis that every element of $P$ has both an immediate predecessor and an immediate successor (unless it is the minimum or the maximum of $P$). Are conclusions (1) and (4) still true?
\end{quest}

Note that it is sufficient to prove Theorem~\ref{discrete_R} for the case in which every element of $P$ is positive (by dividing $P$ into its positive and negative parts). So from now until the end of the proof of Theorem~\ref{discrete_R}, we fix some infinite discrete $P \subseteq R$ which is definable in some strong structure on $R$, which is an Archimedean ordered Abelian group.

We will prove Theorem~\ref{discrete_R} in a series of lemmas: first part (1) as Lemma~\ref{dP_finite}, then (2) and (3) follow from Corollary~\ref{commens} and Proposition~\ref{ev_periodic}, and finally (4) is Proposition~\ref{discrete_images}.

\textbf{From now until the end of this section, we always work in the fixed Archimedean model $\mathcal{R}$.} This means that we use ``$P$'' to mean the interpretation of $P$ in the universe $R$ (not in some saturated extension) and likewise all other definable sets mentioned below are, by abuse of notation, identified with their interpretations in $\mathcal{R}$.

We record an elementary fact which will be used later:

\begin{fact}
\label{discrete_sums}
If $F \subseteq R^{>0}$ is finite, then the set of all finite sums of the form $$n_1 f_1 + \ldots + n_k f_k$$ where $n_i \in \N$ and $f_i \in F$ is discrete.
\end{fact}

\begin{lem}
\label{omega_type}
$P$ has order type $\omega$.
\end{lem}

\begin{proof}
If $\mc{R}$ is discretely ordered, then this follows immediately from the fact that $\mc{R}$ is Archimedean. So assume $\mc{R}$ is densely ordered. Since $\mc{R}$ is Archimedean, we may assume that $(R, <, +)$ is embedded in $(\R, <, +)$ as an ordered subgroup.

Since $P \subseteq \R^{>0}$ is infinite, there is a subset $P_0 \subseteq P$ which is order-isomorphic to $\omega$ or to $\omega^*$, the reverse order type of $\omega$. Suppose that $(P_0, <) \cong \omega^*$. Then since $P_0$ is bounded below by $0$, $\alpha = \inf(P_0) \in \R$. For any $\epsilon > 0$ and $a \in R$ we can define $$X_{a,\epsilon} := \{|a - x | : x \in P \textup{ and } |a - x| < \epsilon\},$$ and $X_{a,\epsilon} \subseteq (0, \epsilon)$, and by choosing $a \in R$ sufficiently close to $\alpha$ we can ensure that $X_{a,\epsilon}$ is infinite; but this contradicts Theorem~\ref{small_discrete}.

So there is some $P_0 \subseteq P$ which is order-isomorphic to $\omega$. If $P_0$ were not cofinal in $R$, then repeating the same argument as above with $\alpha = \sup(P_0) \in \R$, we would obtain a contradiction, so $P_0$ is cofinal in $R$ and hence in $P$. Similarly, there can only be finitely many elements of $P$ between two consecutive elements of $P_0$ (or between the first element of $P_0$ and $0$), since otherwise we would have a bounded subset $P_{00} \subseteq P$ order-isomorphic to $\omega$ or $\omega^*$, yielding the same contradiction as before.

\end{proof}

\begin{definition}
Let $\{a_i : i \in \N\}$ be an enumeration of $P$ in increasing order.
\begin{enumerate}
\item Let $s_P : P \rightarrow P$ be the function $s_P(a_i) = a_{i+1}$. 
\item Given any $f: P \rightarrow R$, let $\Delta f : P \rightarrow R$ be the function $\Delta(f)(x) = f(s_P(x)) - f(x)$.
\item $\Delta P = \{s_P(x) - x : x \in P\}$.
 \end{enumerate}
\end{definition}

In the sequel, we will apply the same notation $s_{P'}$ and $\Delta_{P'}$ to other infinite discrete definable sets $P' \subseteq R^{>0}$, which is warranted in light of Lemma~\ref{omega_type}.

\begin{lem}
\label{dP_bounded}
$\Delta P$ is bounded.
\end{lem}

\begin{proof}
Assume that $\Delta P$ is unbounded. 

\begin{claim}
There is some infinite definable $P' \subseteq P$ such that $$\lim_{n \rightarrow \infty} \Delta s_{P'}(n) = \infty.$$
\end{claim}

\begin{proof}
Let $$P' = \{x \in P : \forall y \in P \, \left[ y < x \Rightarrow s_P(y) - y < s_P(x) - x  \right] \}.$$ Then $\Delta s_{P'}(x) \geq \max \{ \Delta s_P(y) : y \leq x \}$ and the desired property follows.

\end{proof}

Replacing the original $P$ with $P'$ as in the Claim, without loss of generality $\lim_{n \rightarrow \infty} \Delta s_{P}(n) =\infty$. But this immediately contradicts Corollary~\ref{discr_sparse} above, so we are done.

\end{proof}

\begin{lem}
\label{dP_families}
Suppose that $\{P_a : a \in X\}$ is an infinite definable family of subsets $P_a$ of $P$, where $X \subseteq R^m$. Then there is a $K \in R$ such that for every $a \in X$, if $P_a$ is infinite then $\lim \sup \Delta P_a \leq K$.
\end{lem}

\begin{proof}
Suppose towards a contradiction that there is no such $K$. Then for any $K \in R$, there is always some $a \in X$ such that the set $$P_{a,K} := \{x \in P_a : s_{P_a}(x) - x > K\}$$ is infinite, and clearly every element of $\Delta P_{a,K}$ is greater than $K$.

Then we will contradict Lemma~\ref{discrete_family} by finding, in some elementary extension of $\mc{R}$, elements $\{a_i \in X : i \in \N\}$ and $K_i > 0$ such that:

\begin{enumerate}[i]
\item For every $n \in \N$, $|P_{a_0, K_0}| \geq n$; and
\item For every $n \in \N$, $|(0, K_i) \cap 3 \cdot  P_{a_{i+1}, K_{i+1}} | \geq n$.
\end{enumerate}

(Note that taking $D_i = P_{a_i, K_i}$ and $\varepsilon_i = K_i$, condition (2) of Lemma~\ref{discrete_family} is automatic.)

Let $\Gamma$ be the type expressing these properties of the parameters $a_i, K_i$, and let $\Gamma_0 \subseteq \Gamma$ be a finite subset which only mentions parameters with $i \leq N$ and such that $N$ also bounds all the $n$ from formulas of type (ii) in $\Gamma_0$. Start by picking $K_N > 0$ arbitrarily and then $a_N \in X$ such that $P_{a_n, K_N}$ is infinite. Given $a_{i+1}$ and  $K_{i+1}$ for $i > 0$ such that $P_{a_{i+1}, K_{i+1}}$ is infinite, first pick $K_i$ large enough that $|(0, K_i) \cap 3 \cdot P_{a_{i+1}, K_{i+1}} | \geq N$, then pick $a_i \in X$ such that that $P_{a_i, K_i}$ is infinite.

\end{proof}

\begin{lem}
\label{dP_finite_limits}
$\Delta P$ has finitely many limit points.
\end{lem}

\begin{proof}

Suppose not, and let $L$ be the set of all limit points of $\Delta P$. If $0$ is the unique limit point of $L$, then we immediately contradict Corollary~\ref{limitpts}.

Otherwise, $L$ either has no limit points and is unbounded in $R$, or else $L$ has a limit point greater than $0$; in either of these cases, we can choose an $\epsilon > 0$ in $R$ and an infinite collection $S = \{a_i : i \in \N\}$ of limit points of $\Delta P$ such that $\inf(S) \geq \epsilon$ ($S$ does not need to be definable).  Then we pick a pairwise disjoint collection of intervals $\{(\ell_i, r_i) : i \in \N \}$ such that $a_i \in (\ell_i, r_i)$ and $\ell_i > \epsilon/2$ for every $i$. Let $$P_{\ell_i, r_i} = \{x \in P : \ell_i < s_P(x) - x < r_i \}.$$ The sets $P_{\ell_i, r_i}$ are definable, pairwise disjoint, and infinite.

By Lemma~\ref{dP_families}, there is some $K \in R$ such that $\lim \sup (\Delta P_{\ell_i, r_i}) < K$ for every $i$. Pick $n \in \N$ large enough that $n \cdot (\epsilon/2) > K$ (here we are using the assumption that we are are working in a fixed Archimedean model). By our assumptions, for each $i \in \N$, $$P_{\ell_i, r_i} \cap [j \cdot K, (j+1) \cdot K] \neq \emptyset$$ holds for all but finitely many $j \in \N$. Therefore there is some $j$ such that more than $n$ of the sets $P_{\ell_i, r_i}$ intersect the interval $[j \cdot K, (j+1) \cdot K]$, which has length $K$. But any two distinct elements of $\cup_{i \in \N} P_{\ell_i, r_i}$ are a distance of at least $\epsilon/2$ apart, so this is absurd.

\end{proof}

\begin{cor}
\label{dP_finite}
$\Delta P$ is finite.
\end{cor}

\begin{proof}

We already know that $\Delta P$ is bounded by Lemma~\ref{dP_bounded}, so it suffices to show that $\Delta P$ has no limit points. By Lemma~\ref{dP_finite_limits}, the only case to consider is when $\Delta P$ has a finite, nonzero number of limit points. But then each of these limit points would be a limit point of $\Delta P \setminus (\Delta P)'$, contradicting Corollary~\ref{limitpts}.

\end{proof}

\begin{cor}
\label{commens}
Say $X_1, X_2$ are two infinite arithmetical sequences contained in $P$ (not necessarily definable). Then $X_1$ and $X_2$ are \emph{commensurable}: that is, there is a single arithmetic sequence $Y$ which contains both $X_1$ and $X_2$ as subsequences.
\end{cor}

\begin{proof}
Otherwise, for every $\varepsilon > 0$, the set $\Delta P$ would contain an element in $(0, \varepsilon)$, contradicting Lemma~\ref{dP_finite}.
\end{proof}

\begin{definition}
\label{motif}
Given a finite sequence $\sigma =  \langle c_1, \ldots c_n \rangle$ of elements of $\Delta P$, we define $$P_\sigma = \{a \in P : \forall i \left[ 1 \leq i < n  \Rightarrow s^i_P(a) - s^{i-1}_P(a) = c_i \right] \}.$$ We say that $\sigma$ is \emph{infinitely recurring} if $P_\sigma$ is infinite.

\end{definition}

\begin{lem}
\label{inf_recurrence}
There is some fixed $m \in \N$ such that for any sufficiently large $n \in \N$, there are precisely $m$ infinitely-recurring sequences of length $n$.
\end{lem}

\begin{proof}
First note that it is enough to show that there is some fixed $m$ such that for any $n \in \N$, there are at \emph{at most} $m$ sequences of length $n$ such that $P_{\sigma}$ is infinite: this is because any  $\sigma = \langle c_1, \ldots, c_n \rangle$ such that $P_{\sigma}$ is infinite extends to at least one $\sigma' = \langle c_1, \ldots, c_{n+1}\rangle$ of length $(n+1)$ such that $P_{\sigma'}$ is infinite (since $\Delta P$ is finite and by the pigeonhole principle).

The general idea is to apply Lemma~\ref{dP_families} above to the sets $P_{\sigma}$, but we cannot do this directly since they may not constitute a single uniformly definable family.

So for $(b,c)$ and $(b',c') \in P^2$, say $(b,c) \sim (b',c')$ if the interval $[b,c]$ is ``isomorphic via translation'' to $[b', c']$: that is, $c' - b' = c - b$ and $$\forall x \in [b,c] \left( x \in P \Leftrightarrow x + b' - b \in P \right).$$ Then let $$P_{(b,c)} = \{ a \in P : (b,c) \sim (a, a + c - b) \}.$$ 

Note that for any finite sequence $\sigma$ from $\Delta P$, the set $P_{\sigma}$ equals $P_{(b,c)}$ for some $(b,c) \in P^2$.

By Lemma~\ref{dP_families}, there is a $K \in R$ such that for every $(b,c) \in P^2$, if $P_{(b,c)}$ is infinite then $$\lim \sup \Delta P_{(b,c)} \leq K.$$ If $m \in \N$ and there are at least $m$ distinct sequences $\sigma_1, \ldots, \sigma_m$ of length $n$ such that the sets $P_{\sigma_i}$ are all infinite, then the sets $P_{\sigma_1}, \ldots, P_{\sigma_m}$ are pairwise disjoint, and it follows that if $k = \min (\Delta P)$ then there must be some $\sigma_i$ such that $\lim \sup (\Delta P_{\sigma_i}) \geq m \cdot k$. Therefore $k \cdot m \leq K $, yielding a finite bound on $m$ since $k$ and $K$ are from the Archimedean structure $\mc{R}$.
\end{proof}

Now use the previous Lemma to fix $n, m \in \N$ such that for every $n' \geq n$, there are precisely $m$ distinct infinitely-recurring sequences of length $n'$. Also, let $\sigma_1, \ldots, \sigma_m$ list all of the infinitely-recurring sequences of length $n$.

\begin{lem}
\label{extensions}
For any infinitely-recurring sequence $\sigma$ from $\Delta P$ of length $n$, there is a unique infinitely-recurring sequence $\sigma'$ of length $n+1$ such that $\sigma'$ extends $\sigma$.
\end{lem}

\begin{proof}
The existence of some such infinitely-recurring extension $\sigma'$ is immediate from the finiteness of $\Delta P$ and the pigeonhole principle. If we choose $\sigma'_i$ to be some infinitely-recurring extension of $\sigma_i$ of length $n+1$, then $\sigma'_1, \ldots, \sigma'_m$ are $m$ distinct infinitely-recurring sequences of length $n+1$, so by the choice of $m$ and $n$, they must list \emph{every} infinitely-recurring sequence of length $n+1$. This establishes the uniqueness of $\sigma'_i$.

\end{proof}

\begin{definition}
\label{successor} Given an infinitely-recurring sequence $\sigma = \langle c_1, \ldots, c_n\rangle$ of length $n$, let $\sigma' = \langle c_1, \ldots, c_{n+1} \rangle$ be the extension defined in Lemma~\ref{extensions}. Then we define $$\sigma^+ = \langle c_2, c_3, \ldots, c_{n+1} \rangle,$$ which is another infinitely-recurring sequence of length $n$.
\end{definition}

Now if we recursively define a sequence $\tau_1, \tau_2, \ldots$ such that $\tau_1 = \sigma_1$ and $\tau_{i+1} = \tau_i^+$, then the fact that there are only finitely many infinitely-recurring sequences of length $n$ implies that $\{\tau_i : i \in \omega\}$ is \emph{eventually periodic}: there is some $\ell$ (the period) and some $N$ such that for every $i \geq N$, we have $d_{i + \ell} = d_i$. If $\tau_i, \ldots, \tau_{i + \ell} = \tau_i$ is a cycle of length $\ell$, then starting from any sufficiently large $a$ realizing $P_{\tau_i}$, a routine inductive argument implies that any $a' \in P$ with $a' \geq a$ must realize one of the predicates $P_{\tau_i}, \ldots, P_{\tau_{i + \ell - 1}}$. It follows from the definition of $m$ that $\ell = m$ and that the sequence $\tau_1, \tau_2, \ldots$ is in fact periodic of periodicity $m$.

By the previous paragraph, we may assume that $\sigma_i^+ = \sigma_{i+1}$ if $i < m$ and $\sigma_m^+ = \sigma_1$.

\begin{prop}
\label{ev_periodic}
If $d_i = a_{i+1} - a_i$, then the sequence $d_0, d_1, \ldots$ is eventually periodic of periodicity $m$.

Thus $P$ is the union of some finite set and finitely many infinite arithmetic sequences, establishing part (2) of Theorem~\ref{discrete_R}.
\end{prop}

\begin{proof}
As observed in the proof of Lemma~\ref{extensions}, $\sigma'_1, \ldots, \sigma'_m$ must enumerate all the infinitely-recurring sequences of length $n+1$. This plus the fact that there are only finitely many sequences from $\Delta P$ of length $n$ or $n+1$ implies that we can choose an $N \in \omega$ such that for every $i \geq N$,

\begin{enumerate}
\item $a_i$ realizes one of the predicates $P_{\sigma_1}, \ldots, P_{\sigma_m}$, and
\item $a_i$ realizes one of the predicates $P_{\sigma'_1}, \ldots, P_{\sigma'_m}$.
\end{enumerate}

So if $i \geq N$ and $a_i$ realizes $P_{\sigma_j}$, then $a_{i+1}$ realizes $P_{\sigma^+_j}$ and $\sigma^+_j$ is either $\sigma_{j+1}$ (if $j < m$) or $\sigma_1$ (if $j = m$). This implies that $\{d_i : i \in \omega\}$ is periodic for $i \geq N$ with periodicity $m$.

\end{proof}

Finally, we establish part (4) of Theorem~\ref{discrete_R}:

\begin{prop}
\label{discrete_images}
If $P \subseteq R^{>0}$ is definable, infinite, and discrete, then for any definable $f: P \rightarrow R$, the image $f(P)$ is also discrete.
\end{prop}

\begin{proof}
Suppose $f(P)$ is not discrete. Then the derivative $(f(P))'$ is nonempty, and by Corollary~\ref{limitpts}, $(f(P))'$ must be infinite, since otherwise each of its finitely many points would be a limit point of $f(P) \setminus (f(P))'$. Therefore we can choose an infinite collection $\{I_i : i \in \N \}$ of pairwise-disjoint intervals such that $f(P) \cap I_i$ infinite for each $i$. Let $P_i = f^{-1}\left[I_i\right]$, which is a definable family of infinite, pairwise-disjoint subsets of $P$.

Say $P = \{a_i : i < \omega\}$ is the enumeration of $P$ in increasing order. Letting $\widetilde{P}_j = \{i \in \N : a_i \in P_j\}$, for every positive integer $n$, all but finitely many of the sets $\widetilde{P}_j$ must have upper density less than $1/n$. But if $k = \min \Delta P$ (which exists by Lemma~\ref{dP_finite}), then this implies that $\lim \sup \Delta P_j \geq n \cdot k $ for all but finitely many $j \in \N$, contradicting Lemma~\ref{dP_families}. 
\end{proof}

\begin{remark}
The proof part (2) of Theorem~\ref{discrete_R} clearly can be applied to any unary definable set $P \subseteq R$ in a strong ordered Abelian group $ \langle R; <, +, \ldots \rangle$ such that $P \cap R^{>0}$ has an initial segment $P_0 = \{a_i : i < \omega\}$ of order type $\omega$, and such that there is a single Archimedean class $C$ which contains every $a_i$ and every $a_{i+1} - a_i$; in this case, the conclusion is that $P_0$ is a finite union of arithmetic progressions. In fact, it is enough just to assume that every $a_{i+1} - a_i$ is in the same Archimedean class $C$, since (working via translations) we may as well assume that every $a_i$ is in $C$ as well.
\end{remark}

We also note in passing that there is a theorem in algebra by Conrad (see \cite{Conrad}) which says that any Archimedean left-ordered group (that is, a group equipped with an ordering which is invariant under left translations) embeds as an ordered group into $\langle \R; +, < \rangle$ and hence is Abelian. Therefore Theorem~\ref{discrete_R} applies of any Archimedean left-ordered group with a strong theory. There do seem to be examples of dp-rank $2$ left-ordered non-Abelian groups, such as the Klein bottle group, which we may discuss in a future paper, and it may be interesting to investigate whether Theorem~\ref{discrete_R} can be generalized to left-ordered groups with strong theories.

\section{Examples}

In this section we provide a series of examples that illustrate the large variety of possibilities for structures of finite inp-rank or dp-rank.

\subsection{$\Th(\langle \R; +, <, 0,1 ,\Z, \Q \rangle)$}

Corollary \ref{nowheredense} states that in a strong densely-ordered expansion of an Abelian group there are no definable infinite discrete sets with an accumulation point.  In this subsection we show by example that it is possible to have a such a structure with an infinite definable discrete set, and furthermore that this is possible in a definably complete structure.  Also we would like to have an example of a finite dp-rank structure with a definable dense and codense subset.  We show that in fact we may have both of types of sets simultaneously in a finite dp-rank structure. 

We show that:

\begin{prop}\label{zq}
$\Th(\langle \R; +, <, \Z, \Q \rangle)$ has dp-rank $3$.
\end{prop}

Let $\mc{L}=\{+,<,0,1,Z,Q,\floor{\;}, \lambda\}_{\lambda \in \Q}$ where $Z,Q$ are unary relations and $\floor{\;}$ and the $\lambda$'s are unary functions.  Consider the $\mc{L}$-structure $\mathcal{R}=\langle \R;  +, <,0, 1,  \Z, \Q, \floor{\; }, \lambda \rangle_{\lambda \in \Q}$ where the $\lambda$ are unary functions for multiplication by a rational and $\floor{\;}$ is the integer part function ($\floor{x}$ is the greatest integer less than or equal to $x$).  Let $T=\Th(\mc{R})$.   Let $\mc{L}_0$ be  $\mc{L}$ with the symbol $Q$ omitted. We need to establish that $\Th(\mathcal{R})$ has quantifier elimination.  Our proof will rely heavily on the fact that the $\mc{L}_0$-structure $\mc{R}_0=\langle \R; +, <,0,1, Z, \floor{\;},\lambda\rangle_{\lambda \in Q}$ has quantifier elimination and is universally axiomatizable (\cite{ivp}).  Let $T_0=\Th(\mc{R}_0)$.     Before proving quantifier elimination we collect  some basic facts, which we state in a form necessary for use in the ensuing results, on terms in models of $T_0$.  The proofs are by straightforward induction on terms.  In the remainder of this section we write $A$ for the universe of $\mathfrak{A}$ and similarly for other Fraktur letters denoting models.

\begin{lem}\label{term} 
 Let $\mf{B} \models T_0$ and let $t(x,\ob{y})$ be a term. Let  $\mathfrak{A} \subseteq \mf{B}$ and $l \in Z(A)$.  

\begin{enumerate} \item  $t(x,\ob{y})$ is of the form 
\[t(x,\ob{y})=\sum^m_{i=1}\lambda_i\floor{s_i(x,\ob{y})}+\lambda x + s(\ob{y}).\] 
where the $s_i$ and $s$ are terms.  In particular if
$\mf{B} \models Z(t(b,\ob{a}))$ with $\ob{a} \subseteq A$ then for some $\mu \in \Q$ and some $a \in A$ it must be the case that 
$\mf{B} \models Z(\mu b +a)$.
\item
Suppose that $\ob{a} \subseteq A$ and let $t=t(x,\ob{a})$. There are intervals $I_1, \dots, I_n \subseteq (l,l+1)$ definable from $A$ so that $t \restriction I_i$ is equal to $\lambda_i x + a_i$ for each $i \in \{1 \dots n\}$ for some $\lambda_i \in \Q$ and $a_i \in A$ and so that $\floor{t} \restriction I_i$ is constant for each $i \in \{1 \dots, n\}$.
\item Let $\ob{a} \subseteq A$ and set $t=t(x,\ob{a})$.  There is a partition of $Z(B)$ into sets $X_1, \dots, X_n$ definable in the structure 
$\langle Z(B),+,< \rangle$ so that $t \restriction X_i=\lambda_ix + b_i$ for some $\lambda_i \in \Q$ and $b_i \in B$ for all $i \in \{1 \dots n\}$.
\end{enumerate}
\end{lem} 

These facts have useful consequences for definable sets in model of $T_0$.  The following lemma is an easy consequence of Lemma \ref{term} and quantifier elimination for $T_0$.

\begin{lem}\label{defsets}.  Let $\mf{B} \models T_0$ and let $X \subseteq B$ be definable over $\mf{A} \subseteq \mf{B}$.
\begin{enumerate}
\item If $n \in Z(A)$ then $X \cap (n,n+1)$ is a finite union of points and intervals definable from $A$.
\item If $X \subseteq Z(B)$ then $X$ is definable in the structure $\langle Z(B),+,< \rangle$.
\end{enumerate}
\end{lem}

\begin{prop} $T=\Th({\mc{R}})$ has quantifier elimination.
\end{prop}

\begin{proof}  
 For notation we write $\qftp(a)$ for the quantifier free type of $a$ in $\mc{L}$ and $\qftp_0(a)$ be the quantifier free $\mc{L}_0$-type.

We use a standard embedding test.  Let $\mf{B}$ and $\mf{B}^{\prime}$ be models of $T$ with $\mf{B }^{\prime}$ $\omega$-saturated and suppose that $\mf{A}$ is a substructure of both $\mf{B}$ and $\mf{B}^{\prime}$ which is finitely generated by $\ob{a}$.  Note that as $T_0$ eliminates quantifiers and is universally axiomatizable we have that $\mf{A}\restriction_{ \mc{L}_0} \preceq \mf{B}\restriction _{\mc{L}_0}$ and $\mf{A}\restriction _{\mc{L}_0}\preceq \mf{B}^{\prime}\restriction _{\mc{L}_0}$.  For $b \in B \setminus A$ we 
must find $b^{\prime} \in B^{\prime}$ so that $\qftp(b^{\prime}/\ob{a})=\qftp(b/\ob{a})$.

First suppose that $b \in Z(B)$.  As $T_0$ eliminates quantifiers we may find $b^{\prime} \in \mf{B}^{\prime}$ so that $\qftp_0(b/\ob{a})=\qftp_0(b^{\prime}/\ob{a})$.  We claim that in fact $\qftp(b/\ob{a})=\qftp(b^{\prime}/\ob{a})$.  
To verify this we need only check that if $t(x,\ob{y})$ is a term then $\mf{B} \models Q(t(b,\ob{a}))$ if and only if $\mf{B}^{\prime} \models Q(t(b^{\prime},\ob{a}))$.  By Lemma \ref{term} $t(x,\ob{y})$ is of the form:
\[t(x,\ob{y})=\sum^m_{i=1}\lambda_i\floor{s_i(x,\ob{y})}+\lambda x + s(\ob{y}).\] 
Thus $\mf{B} \models Q(t(b,\ob{a}))$ if and only if $b+\frac{1}{\lambda}s(\ob{a}) \in Q(B)$.  But this holds if and only if $\frac{1}{\lambda}s(\ob{a}) \in Q(B)$ if and only if $\frac{1}{\lambda}s(\ob{a}) \in Q(\mf{B}^{\prime})$.  Which finally holds if and only if $\mf{B}^{\prime} \models Q(t(b^{\prime},\ob{a}))$.

Hence we may assume that $b \notin Z(B)$ and further more we may without loss of generality assume that $\floor{b} \in A$.  Thus note that by Lemma \ref{defsets} if $\varphi(x) \in \qftp_0(b/\ob{a})$ then $\varphi(B)$ has non-empty interior.
In particular if $a \in A$ the coset $a+Q(B)$ has non-trivial intersection with $\varphi(B)$ and the same holds in $\mf{B}^{\prime}$.

Thus if for some $a \in A$ it is the case that $\mf{B} \models Q(b+a)$  we may find $b^{\prime} \in B^{\prime}$ so that $b^{\prime}$ realizes $\qftp_0(b/\ob{a})$ and so that $\mf{B}^{\prime} \models Q(b^{\prime}+a)$.  We claim that $\qftp(b/\ob{a})=\qftp(b^{\prime}/\ob{a})$.  As above we are reduced to checking that $\mf{B} \models Q(\lambda b +c)$ if and only if $\mf{B}^{\prime} \models Q(\lambda b^{\prime} + c)$ where $c \in A$.  We have that 
$\lambda b + c \in Q(B)$ if and only if $b \in Q(B)-\frac{1}{\lambda}c$ but as $b \in Q(B)-a$ this holds if and only if $\frac{1}{\lambda}c-a \in Q(B)$.  The previous clause holds if and only if $\frac{1}{\lambda}c-a \in Q(B^{\prime})$ which holds if and only if $\mf{B}^{\prime} \models Q(\lambda b + c)$.

The final case is when $b \notin Q(B)+a$ for any $a \in A$.  The proof is much the same as the previous case simply using the fact that the union of finitely many complements of cosets of $Q(B)$ (or $(Q(B^{\prime})$) is codense in the line.

\end{proof}

So far in considering models of $T$ we have focused on their $\mc{L}_0$ reducts.  Now we need to consider the structure \[\mc{R}_1=\langle \R, +, <, 0,1,\Q, \lambda\rangle_{q \in \Q}.\]  Let $\mc{L}_1$ be its language and $T_1$ its theory.  Note that by \cite{opencore2} $T_1$ has quantifier elimination.  In particular if $\mf{B} \models T_1$ then any definable subset of $B$ is a finite union of points, interval, intervals intersected with a coset of $Q(B)$, or intervals intersected with the union of the  complements of finitely many cosets of $Q(B)$.   For convenience let as call these basic definable sets cells.  Note also that by results from \cite{densepair} if $\mf{B} \models T_1$ then the induced structure on $Q(B)$, and hence on any coset of $Q(B)$, is weakly o-minimal.  We have:

\begin{lem}\label{qdp}  $T_1$ has dp-rank $2$.
\end{lem}

\begin{proof}  Note that by results from \cite{lovely} $T_1$ does not have the independence property thus by results from \cite{adler_strong_dep} it suffices to show that $T_1$ has burden $2$.  Fix $\C$ a saturated model of $T_1$.  It is trivial to see that the burden of $T_1$ is at least $2$ by simply considering disjoint intervals and cosets of $Q(C)$.  

Suppose for contradiction that there were an inp-pattern of depth $3$.  Thus we have formulae $\varphi_i(x, \ob{y})$
for $0 \leq i \leq 2$ and parameters $\ob{a}_i^j$ with $0 \leq i \leq 2$ and $j \in \omega$ witnessing this.  By results from \cite{DGL} we may assume that each $\varphi_i$ defines a cell.  First suppose that one of the formulae, say $\varphi_0$ defines an interval intersected with a coset.  Say that $\varphi_0(x, \ob{a}^0_0)$ defines $I \cap Q(C)+c$.  It follows that the other two rows of the inp-pattern would witness that the induced structure on $Q(C)+c$ would have burden at least $2$, contradicting that weakly o-minimal structures are dp-minimal (see \cite{DGL}).

Since not all of the formulae can define intervals we may assume that $\varphi_0$ defines an interval intersected with the complement of finitely many cosets.  For each $j \in \omega$ let $I^j$ be the interior of the closure of 
$\varphi_0(C, \ob{a}_0^j)$.  Note that in order for the collection $\{\varphi_0(x, \ob{a}_0^j): j \in \omega\}$ to be inconsistent the collection $\{I^j: j \in \omega\}$ must be inconsistent.  Hence in our inp-pattern we can replace the top row with one consisting entirely of intervals.  We can then do the same for the following two rows, obtaining an inp-pattern consisting only of intervals, a contradiction.
\end{proof}

Via quantifier elimination we also have an easy analogue of  part 1  from Lemma \ref{defsets}.

\begin{lem}\label{qdef} Let $\mf{B} \models T$ and suppose that $X \subseteq B$ is definable and $n \in Z(B)$.  Then $X \cap (n,n+1)$ is definable in $\mf{B} \restriction \mc{L}_1$.
\end{lem}

\begin{proof} As 1. from Lemma \ref{defsets}.  Just note that Lemma \ref{term} holds in models of $T$ as well as in models of $T_0$.
\end{proof}

In order to establish that $T$ has dp-rank $3$ we need a basic fact which follows from Proposition 4.20 in \cite{simonnip}.

\begin{fact}\label{coord}  Let $T$ be any theory and $\C$ a saturated model of $T$.  Let $a \in C$ and $\ob{b} \in C$.  The 
dp-rank of $\tp(a)$ is bounded by the sum of the dp-rank of $tp(\ob{b})$ and the dp-rank of $tp(a/\ob{b})$.
\end{fact}

\noindent {\bf Proof of Proposition \ref{zq}:}  Let $\C$ be a large saturated model of $T$.  First we show that there is an inp-pattern of depth three.  Let $\{a_i : i \in \omega\}$ be distinct elements of $Z(\mf{C})$ and let $\varphi_0(x,a_i) := ``x \in (a_i, a_i+\frac{1}{2})"$.  Note of course that the $\varphi_0(x,a_i)$ are pairwise inconsistent.  
Next pick pairwise disjoint open intervals $\{(l_i,r_i) : i \in \omega\}$ each contained in $(0,\frac{1}{2})$.  Let $\varphi_1(x,l_ir_i) :=``x-\floor{x} \in (l_i,r_i)"$.   Once again the $\varphi_1(x,l_ir_i)$ are pairwise inconsistent and if $i,j \in \omega$ then 
$\varphi_0(x,a_i) \wedge \varphi_1(x, l_jr_j)$ defines a non-empty open set.  Finally 
pick $c_i \in \C$ so that $Q(\mf{C}) + c_i$ are distinct cosets and let $\varphi_2(x,c_i) :=``x \in Q+c_i"$.  It is now easy to see that the $\varphi_0(x,a_i), \varphi_1(x,l_ir_i)$, and $\varphi_2(x,c_i)$ form an inp-pattern of depth three.

  Let $c \in C$.     We must show that $\tp(c)$ has dp-rank at most $3$.  First consider $\floor{c}$ by Part 2 of Lemma \ref{defsets} and the dp-minimality of Presburger arithmetic (see \cite{conv-val}), $\tp(\floor{c})$ has dp-rank at most $1$.  Now by Lemmas \ref{qdp} and \ref{qdef} $\tp(c/\floor{c})$ has rank at most $2$.  Hence by Fact \ref{coord} $\tp(c)$ has rank at most $3$. \qed

\smallskip

Notice that essentially an identical proof establishes the fact that $T_0$ has dp-rank $2$.

\subsection{DOAGs with dense graphs}

In this subsection we show that there are definably complete structures of finite dp-rank in which there are definable unary functions whose graphs are dense in the plane.  This situation was considered in \cite{opencore2} where the absence of such functions is shown to imply a good degree of tractability in structures of o-minimal open core.  Thus our examples indicate that even under the assumption of finite dp-rank we can not hope that definable functions are quite so tractable.  In this subsection we also construct structures of dp-rank equal to $n$ for any  $n \geq 1$.

Let $H$ be a Hamel basis for the reals over the rationals.  For $h^* \in H$ let $\pi_{h^*}$ be projection on $h$ namely the function that for 
$a=\Sigma_{h \in H}a_h \cdot  h$ in $\R$ maps $a$ to $a_h \cdot h$.  Fix $h_1 \dots, h_n \in H$ and for convenience denote $\pi_{h_n}$ by $\pi_n$.  Let 
$\mf{M}_n$ be the structure $\langle \R, +,0,1, <, \pi_1, \dots, \pi_n, \lambda \rangle_{\lambda \in \Q}$  where the $\lambda$'s are unary functions for multiplication by $\lambda$.  Let $\mc{L}$ be the language in which this structure is presented and let $\mc{L}_0$ be $\mc{L}$ without the $\pi_i$'s.  Let $T_n=Th(\mf{M}_n)$.

We begin by noting that the functions $\pi_h$ are "wild":

\begin{lem}\label{densegraph}  The graph of $\pi_h: \R \to \R$ is dense in $\R^2$.
\end{lem}

\begin{proof}  This follows easily from the density of $H$.
\end{proof}

Nonetheless the theories $T_n$ are well-behaved.

\begin{prop} $T_n$ eliminates quantifiers.
\end{prop}

\begin{proof}  First note that any term $t$ considered to be in a variable $x$  is equivalent to a term of the form:
\[\lambda_1\pi_1(x) +\ \dots + \lambda_n\pi_n(x)+\mu x + s\] where the $\lambda_i$ and $\mu$ are rational numbers and $s$ is an $\mc{L}$-term that does not involve $x$.

Thus we can reduce to the case where we need to eliminate the existential quantifier from a formula of the form:

\[\exists x (\bigwedge_{i=1}^k \sum_{r=1}^n \lambda^i_r\pi_r(x)+\mu^i x < s^i \wedge \bigwedge_{j=1}^l  \sum_{r=1}^n\lambda^j_r\pi_r(x)+\mu^j x = s^j)\]

But this is equivalent to:

\begin{align*}
\exists x_1 \dots x_{n+1}( \bigwedge_{r=1}^n&(\pi_r(x_r)=x_r \wedge \pi_r(x_{n+1})=0) \\  \wedge   \bigwedge_{i=1}^k&\sum_{r=1}^n \lambda^i_rx_r + \mu^i(x_1 + \dots + x_{n+1}) < s^i \\  \wedge \bigwedge_{j=1}^l& \sum_{r=1}^n\lambda^j_rx_r+\mu^j(x_1+ \dots +x_{n+1})=s^j).
\end{align*}

Thus we are reduced to eliminating the quantifiers from formulae of the form:
\begin{align*}
\exists x_1\dots x_{n+1}(\bigwedge_{r=1}^n&(\pi_r(x_r)=x_r \wedge \pi_r(x_{n+1})=0)\\ \wedge&\varphi(x_1, \dots, x_{n+1},s_1(\ob{y}), \dots, s_l(\ob{y})))
\end{align*}
where $\varphi(\ob{x},\ob{z})$ is a quantifier free formula in the language $\mc{L}_0$ and the $s_i$ are $\mc{L}$-terms.  We do this by serially eliminating the quantifiers beginning with $\exists x_{n+1}$.  Thus consider the formula:
\begin{equation*}\exists x_{n+1}(\bigwedge_{r=1}^n(\pi_r(x_r)=x_r \wedge \pi_r(x_{n+1})=0) \wedge \varphi(\ob{x},\ob{s}(\ob{y}))).\tag{*}\end{equation*} (Here we write $\ob{s}(\ob{y})$ for $s_1(\ob{y}), \dots, s_l(\ob{y})$.)
Let $\psi(x_1, \dots, x_n, \ob{z})$ be the quantifier-free $\mc{L}_0$-formula expressing 
\[``\varphi(\ob{x},-,\ob{z})\text{ has non-empty interior}".\]  Note that there are $N \in \N$ and  $\mu^j_i \in \Q$ for $i \in \{1, \dots, n\}$ and $j \leq N$ and $\mc{L}$-terms $k_1(\ob{z}), \dots, k_N(\ob{z})$ so that if $\mf{M} \models T_n$ and 
$\mf{M} \models \neg \psi(a_1, \dots, a_n, s_1(\ob{b}), \dots,  s_l(\ob{b}))$ then 
$\mf{M} \models a_{n+1} = \mu^j_1 a_1 + \dots \mu^j_n a_n +k_j(\ob{b})$ for 
some $j \leq N$.

Hence we can reduce to the case where either $\varphi(\ob{x},\ob{s}(\ob{y}))$ implies that 
$\varphi(x_1, \dots, x_n, - , \ob{s}(\ob{y}))$ has interior or $\varphi(\ob{x},\ob{s}(\ob{y}))$ is of the form $x_{n+1}=\mu_1x_1 +\dots + \mu_nx_n +t(\ob{y})$ where $t$ is an $\mc{L}$-term.  

In the first case as the set of all $x_{n+1}$ so that 
$\pi_i(x_{n+1})=0$ for all $1 \leq i \leq n$ is dense in the line the formula 
(*) is always consistent and hence equivalent to 
\[\bigwedge_{r=1}^n(\pi_r(x_r)=x_r) \wedge \psi(\ob{x},\ob{s}(\ob{y})).\]

In the latter case (*) is equivalent to 
\[\bigwedge_{r=1}^n(\pi_r(x_r)=x_r \wedge \mu_rx_r+\pi_{r}(t(\ob{y})))=0.\]

In either case we have reduced to eliminating the quantifiers from a formula of the form:
\[\exists x_1 \dots x_n (\bigwedge_{r=1}^n(\pi_r(x_r)=x_r) \wedge \sigma(x_1, \dots, x_n, t_1(\ob{y}), \dots, t_k(\ob{y}))).\]  Where $\sigma(\ob{x}, \ob{z})$ is an $\mc{L}_0$ formula and the $t_i$'s are $\mc{L}$-terms.  We can now eliminate the variables $x_n$ through $x_1$ serially by essentially the identical method we used to eliminate $x_{n+1}$ noting that for all $1 \leq j \leq n$ the 
set $X_j=\{a: \pi_j(b)=a \text{ for some } b\}$ is dense in the line.
\end{proof}

Before we can establish the dp-rank of $T_n$ we need a basic lemma whose proof is straightforward.

\begin{lem}\label{indh} Let $\C \models T_n$.  \begin{enumerate}
\item Let $a_1 \dots a_n \in C$.  The induced structure on the set \[F(\ob{a})=\{x : \bigwedge_{i=1}^n\pi_i(x)=a_i\}\] is weakly o-minimal.
\item For each $i$ the induced structure on the set \[V_i=\{x : \pi_i(y)=x \text{ for some } y\}\] is weakly o-minimal.

\end{enumerate}
\end{lem}

\begin{prop} 
\label{dense_graphs}
$T_n$ has dp-rank $n+1$.
\end{prop}

\begin{proof}  Fix $\C \models T$.  By Lemma \ref{densegraph}, for each $1 \leq i \leq n$ the graph of $\pi_i$ is dense in the plane.  This fact allows us to easily construct an ict-pattern of depth $n+1$.  We must show that for $c \in C$ that $\tp(c)$ has dp-rank at most $n+1$.  First note that $\tp(c/\pi_1(c) \dots \pi_n(c))\vdash c \in F(\pi_1(c)\dots\pi_n(c))$.  Thus by Lemma \ref{indh} and the dp-minimality of weakly o-minimal structures $\tp(c/\pi_1(c) \dots \pi_n(c))$ has dp-rank at most $1$.  Furthermore $\tp(\pi_i(c))$ has dp-rank at most $1$ for each $1 \leq i \leq n$ by Lemma \ref{indh}.  By the subadditivity of dp-rank (see 
\cite{additivity_dp_rank}) we have that $\tp(\pi_1(c) \dots \pi_n(c))$ has rank at most $n$.  Hence by Fact \ref{coord} $\tp(c)$ has dp-rank at most $n+1$.

\end{proof} 

\subsection{Reducts of Tame Pairs}
In this section we consider an example which is the reduct of a tame pair of real closed fields (as studied in \cite{t-conv}) and show it has dp-rank 2.  Our motivation here is twofold.  First Corollary \ref{nowheredense} establishes that in a strong theory there can not be a definable infinite discrete set with an accumulation point, the current example shows that it is possible to have an infinite definable discrete set in a theory of finite inp-rank which accumulates to a bounded cut.  This also highlights the strength of the  definable completeness assumption.  Secondly Theorem \ref{discrete_R} applies to Archimedean structures of finite inp-rank and shows that discrete definable sets are approximately arithmetic progressions indexed by $\N$, our current example demonstrates that without the Archimedean assumption this fails, namely we have a discrete definable set on which the induced ordering is dense.

Let $R$ be a real closed field which is a proper elementary extension of $\R$.  We consider the following structure:

\[\mc{R}=\langle R, +, <, P, V, st,0,1, \lambda\rangle_{\lambda \in \Q}\]
 where $P$ is a unary predicate for the real numbers, $V$ is a unary predicate for the convex hull of the real numbers, $st$ is the standard part map, and the $\lambda$'s are unary functions for multiplication by $\lambda$.  For convenience we set $st(x)=x$ if $x \notin V$.  Notice that this is simply the additive reduct of a tame pair of real closed fields as studied in \cite{t-conv}.  Let $T=\Th(\mc{R})$.

We need a quantifier elimination result, which though a-priori weaker than those in \cite{t-conv} does not appear to immediately follow from them.

\begin{prop}  $T$ has quantifier elimination.
\end{prop}

\begin{proof}  We let $\mf{B}$ and $\mf{B^*}$ be models of $T$ with $\mf{B}^*$ countably saturated and let $\mf{A}=\langle a_1, \dots, a_n \rangle$ be a finitely generated substructure of both models.  We need to show that if $b \in B$ then there is $b^* \in B^*$ so that $\qftp(b/a_1 \dots a_n)=\qftp(b^*/a_1, \dots, a_n)$.

First suppose that $\mf{B} \models Pb$.  In this case note that the quantifier free type of $b$ is completely determined by the order type of $b$ over $A$.  Thus we readily find $b^*$.

Thus we assume that $\mf{B} \models \neg Pb$ and without loss of generality also assume that $st(b) \in A$.  First let us assume that furthermore $\mf{B} \models Vb$.  In this case we readily see that if $b \notin A$ then $\qftp(b/A)$ is completely determined by the formula $st(x)=st(b)$ and the order type of $b$ over $A$.  Once again we readily find $b^*$.  Finally suppose that $\mf{B} \models \neg Va$.  In this case the quantifier free type of $b$ is determined by its order type over $A$ together with which (if any) of the cosets of the form $V+a$ for $a \in A$ or $P+a$ for $a \in A$ contain $b$.  We can easily find $b^*$.
\end{proof}

\begin{lem}\label{monad}  Let $\mf{C} \models T$.
\begin{enumerate}  
\item The induced structure on $P(C)$ is o-minimal.
\item If $ a \in P(C)$ then the induced structure on $\{x : st(x)=a\}$ is o-minimal.
\item A definable subset of $\neg V(C)$ consists of a finite union of ``cells\rq\rq{} where a cell is either a convex set, a convex set intersect a coset of $P$, or a convex set intersect finitely many complements of cosets of $P$.
\end{enumerate}
\end{lem}

\begin{proof} Immediate via quantifier elimination.
\end{proof}
 
\begin{prop} \label{tame_pair} $T$ has dp-rank $2$.
\end{prop}

\begin{proof}It is immediate from results from \cite{S} that the dp-rank of $T$ is at least $2$.
By results from \cite{deppair} $T$ does not have the independence property, hence we must show that all types $p(x)$ have burden at most $2$.  First of all suppose that $Vx \in p$ and let $a$ realize $p$.  By Lemma \ref{monad} and the dp-minimality of o-minimal structures the type of $st(b)$ has dp-rank $1$.  Again, by Lemma \ref{monad} and the dp-minimality of o-minimal structures $tp(b/st(b))$ has dp-rank $1$.  Hence by Fact \ref{coord} $p(x)$ has dp-rank (or burden) at most $2$.

Now suppose that $\neg Vx \in p$.  From here the proof is almost identical to that of  Lemma \ref{qdp}.  Suppose there were an inp-pattern in $p(x)$ of depth $3$.  Via Lemma \ref{monad} and results from \cite{DGL} we can assume that each row of the pattern consists of cells as described in Lemma \ref{monad}.  If one row consisted of cells containing a coset of $P$ then the remaining two rows of the pattern would induce an inp-pattern of depth $2$ on this coset, but by Lemma \ref{monad} the induced structure on the coset is o-minimal and hence this is impossible.  Hence the rows most consist of convex sets intersected with the complements of finitely many cosets of $P$.  But in order for this to be an inp-pattern the convex sets on each row must be pairwise disjoint.  Hence we would have an inp-pattern of depth $3$ consisting entirely of convex sets, which is impossible.    
\end{proof}

\subsection{Generic Expansions of O-minimal Theories}

In this subsection we show that if $T$ is o-minimal and $T_G$ is the expansion of $T$ with a new generic unary predicate (as constructed in \cite{cp}) then $T_G$ is inp-minimal.  (See \cite{chernikov} for a similar result in the context of $NTP_2$ theories.)  Thus, if $T$ extends the theory of divisible ordered Abelian groups, this gives a good example of an expansion of a divisible ordered Abelian group which is inp-minimal but not dp-minimal (recall that  by a result from \cite{cp} $T_G$ will have the independence property).   In particular  in models of $T_G$ there will be infinite definable dense and codense sets, which is in marked contrast to the dp-minimal case where, as shown in \cite{S}, any infinite definable set most have non-empty interior.  In the specific case where  $T$ is the theory of real closed fields we have an example of an inp-minimal expansion of a real closed field that is not weakly o-minimal.  

Let $T$ be o-minimal in a  language $\mc{L}$.  Let $G$ be a new unary predicate and let $T_G$ be the theory where $G$ is a generic predicate as in \cite{cp} (see also \cite{sergio} for the specific situation where $T$ is o-minimal).   Let $T^*$ be any completion of $T_G$.  We show:

\begin{prop}  
\label{genexpansion}
$T^*$ has burden one.
\end{prop}

\begin{proof}  Fix $\mf{C}$ a large model of $T^*$.   As $T$ is o-minimal and by the quantifier elimination for $T_G$ (see \cite{cp}) for any formula $\varphi(x, \ob{a})$ with parameters is equivalent to a disjunction of formulae of the form:

\[x \in I(\ob{a}) \wedge \bigwedge_{l=1}^nG(f_l(x,\ob{a})) \wedge \bigwedge_{k=1}^{m}\neg G(g_k(x, \ob{a}))\]
where $I(\ob{a})$ is either an open interval or a point, the $f$'s and $g$'s are definable functions continuous and monotone on $I(\ob{a})$ and so that for any $c \in I(\ob{a})$ the values $\{f_1(c, \ob{a}), \dots, f_n(c, \ob{a}), 
g_1(c,\ob{a}), \dots g_m(c, \ob{a})\}$ are all distinct.  For convenience let us call such a formula a cell.

Suppose for contradiction that there is an inp-pattern of depth $2$.  Hence we find formulae $\varphi_1(x, \ob{y})$ and 
$\varphi_2(x, \ob{y})$ and mutually indiscernible sequences $\{\ob{a}_i : i \in \R\}$ and $\{\ob{b}_i : i \in \R\}$ witnessing this.  Let $r$ be a natural number so that $\{\varphi_1(x, \ob{a}_i) : i \in \R\}$ and $\{\varphi_2(x, \ob{b}_i ): i \in \R\}$ are $r$-inconsistent.  By results from \cite{DGL} we may assume that each $\varphi_i$ is a cell of the form:
\[x \in I_i(\ob{y}) \wedge \bigwedge_{l=1}^{n_i}G(f_l^i(x,\ob{y})) \wedge \bigwedge_{k=1}^{m_i}\neg G(g^i_k(x, \ob{y}))\] for $i \in \{1,2\}$.

Trivially if either $I_1(\ob{a}_0)$ or $I_2(\ob{b}_0)$ is a single point we arrive at  a contradiction, thus both $I_1(\ob{a}_0)$ and $I_2(\ob{b}_0)$ are open intervals.  Furthermore if both $\{I_1(\ob{a}_i) : i \in \R\}$ and $\{I_2(\ob{b}_i) : i \in \R\}$ were inconsistent we would have an inp-pattern of depth $2$ consisting exclusively of open intervals which is also impossible.  Hence without loss of generality we assume that $\{I_1(\ob{a}_i) : i \in \R\}$ is consistent.  If $\varphi_1(x,\ob{y})$ contains no terms of the form $G(f(x, \ob{y}))$ or no terms of the form $\neg G(g(x, \ob{y}))$ then by the genericity of $G$ we easily obtain that $\{\varphi_1(x, \ob{a}_i) : i \in R\}$ is consistent.  Hence we assume this is not the case.

   Let $I$ be an open interval so that $I \subseteq \bigcap_{i \in \R}I_1(\ob{a}_i)$.  
   It follows that if $c \in I$ and $i_1< \dots < i_r \in \R$ then for some 
 $1 \leq l \leq n_1$ and $1 \leq s \leq r$ as well as some $1 \leq k \leq m_1$ and some $ 1 \leq t \leq r$ we must have that $f^1_l(c, \ob{a}_{i_s})=g^1_k(c, \ob{a}_{i_t})$ since otherwise by the continuity of all the functions this would fail in a neighborhood of $c$ and thus by the genericity of $G$ we would have that $\{\varphi_1(x, \ob{a}_{i_j}) : 1 \leq j \leq r\}$ is consistent.

Fix $c \in I$. As $T$ is dp-minimal we may find on open interval $J \subseteq \R$ so that in the reduct of $\mf{C}$ to $\mathcal{L}$ the sequence $\{\ob{a}_i : i \in J\}$ is indiscernible over $c$.  Pick $i_1 < \dots < i_{r^*} \in J$ where $r^{*} \geq \max\{r,3\}$.
For notational convenience suppose that $f^1_1(c, \ob{a}_{i_1})=g^1_1(c, \ob{a}_{i_2})$.  Hence 
by indiscernibility we have that $g^1_1(c, \ob{a}_{i_2})=g^1_1(c, \ob{a}_{i_3})$ and $f_1^1(c, \ob{a}_{i_2})=g^1_1(c, \ob{a}_{i_3})$.  Thus $f^1_1(c, \ob{a}_{i_2})=g^1_1(c, \ob{a}_{i_2})$ which which violates the assumption that  $f^1_1(x, \ob{a}_{i_2})$ and $g^1_1(x, \ob{a}_{i_2})$ take on distinct values on all of $I_1(\ob{a}_2)$.  We have arrived at a contradiction and hence $T^*$ has burden one.

\end{proof}

\bibliography{modelth-oag}

\begin{bibdiv}
\begin{biblist}

\bib{adler_strong_dep}{unpublished}{
      author={Adler, Hans},
       title={Strong theories, burden, and weight},
        date={2007},
        note={available on author's website},
}

\bib{DGL}{article}{
      author={Alfred~Dolich, John~Goodrick},
      author={Lippel, David},
       title={Dp-minimal theories: basic facts and examples},
        date={2011},
     journal={Notre Dame Journal of Formal Logic},
      volume={52},
      number={3},
       pages={267\ndash 288},
}

\bib{lovely}{article}{
      author={Berenstein, Alex},
      author={Dolich, Alfred},
      author={Onshuus, Alf},
       title={The independence property in generalized dense pairs of
  structures},
        date={2011},
     journal={Journal of Symbolic Logic},
      volume={76},
       pages={391\ndash 404},
}

\bib{cp}{article}{
      author={Chatzidakis, Zo\'{e}},
      author={Pillay, Anand},
       title={Generic structures and simple theories},
        date={1998},
     journal={Annals of Pure and Applied Logic},
      volume={95},
       pages={71\ndash 92},
}

\bib{chernikov}{article}{
      author={Chernikov, Artem},
       title={Theories without the tree property of the second kind},
        date={2014},
     journal={Annals of Pure and Applied Logic},
      volume={162(2)},
       pages={695\ndash 723},
}

\bib{Conrad}{article}{
      author={Conrad, Paul},
       title={Right-ordered groups},
        date={1959},
     journal={Michigan Mathematics Journal},
      volume={6},
      number={3},
       pages={276\ndash 275},
}

\bib{opencore2}{article}{
      author={Dolich, Alfred},
      author={Miller, Chris},
      author={Steinhorn, Charles},
       title={Structures having o-minimal open core},
        date={2010},
     journal={Transactions of the American Mathematical Society},
      volume={362},
       pages={1371\ndash 1411},
}

\bib{densepair}{article}{
      author={Dries, Lou Van~Den},
       title={Dense pairs of o-minimal structures},
        date={1998},
     journal={Fundamenta Mathematicae},
      volume={157},
       pages={61\ndash 78},
}

\bib{t-conv}{article}{
      author={Dries, Lou Van~Den},
      author={Lewenberg, Adam~H.},
       title={{$T$}-convexity and tame extensions},
        date={1995},
     journal={The Journal of Symbolic Logic},
      volume={60},
       pages={74\ndash 102},
}

\bib{forservi}{unpublished}{
      author={Fornasiero, Antongiulio},
      author={Servi, Tamara},
       title={Definably complete and {B}aire structures and {P}faffian
  closure},
        note={{arXiv:} 0803.3560},
}

\bib{JSW}{unpublished}{
      author={Franziska~Jahnke, Pierre~Simon},
      author={Walsberg, Erik},
       title={Dp-minimal valued fields},
        date={2015},
        note={arXiv:1507.03911},
}

\bib{sergio}{article}{
      author={Fratarcangeli, Sergio},
       title={Elimination of imaginaries in expansions of o-minimal structures
  by generic sets},
        date={2005},
     journal={Journal of Symbolic Logic},
      volume={70},
      number={4},
       pages={1150\ndash 1160},
}

\bib{Goodrick_dpmin}{article}{
      author={Goodrick, John},
       title={A monotonicity theorem for dp-minimal densely ordered groups},
        date={2010},
     journal={Journal of Symbolic Logic},
      volume={75},
      number={1},
       pages={221\ndash 238},
}

\bib{conv-val}{article}{
      author={Guingona, Vincent},
      author={Flenner, Joseph},
       title={Convexly orderable groups and valued fields},
        date={2014},
     journal={The Journal of Symbolic Logic},
      volume={79(1)},
       pages={154\ndash 170},
}

\bib{deppair}{article}{
      author={{G\"{u}nayd{\i}n}, Ayhan},
      author={Hieronymi, Philipp},
       title={Dependent pairs},
        date={2011},
     journal={Journal of Symbolic Logic},
      volume={76(2)},
       pages={377\ndash 390},
}

\bib{Hempel_Onshuus}{article}{
      author={Hempel, Nadja},
      author={Onshuus, Alf},
       title={Groups in $\textup{NTP}_2$},
        date={2015},
        note={arXiv:1510.01049v1},
}

\bib{phillipdisc}{article}{
      author={Hieronymi, Philipp},
       title={Defining the integers in expansions of the real field by closed
  discrete sets},
        date={2010},
     journal={Proceedings of the American Mathematical Society},
      volume={138},
       pages={2163\ndash 2168},
}

\bib{additivity_dp_rank}{article}{
      author={Itay~Kaplan, Alf~Onshuus},
      author={Usvyatsov, Alexander},
       title={Additivity of the dp-rank},
        date={2013},
     journal={Trans. Amer. Math. Soc.},
      volume={365},
       pages={5783\ndash 5804},
}

\bib{Johnson}{unpublished}{
      author={Johnson, Will},
       title={On dp-minimal fields},
        date={2015},
        note={arXiv:1507.02745},
}

\bib{ADHMS}{article}{
      author={Matthias~Aschenbrenner, Deirdre Haskell Dugald~Macpherson,
  Alfred~Dolich},
      author={Starchenko, Sergei},
       title={Vapnik-chervonenkis density in some theories without the
  independence property, i},
     journal={Transactions of the American Mathematical Society},
        note={to appear},
}

\bib{mv}{article}{
      author={Michaux, Christian},
      author={Villemaire, Roger},
       title={Presburger arithmetic and recognizability of sets of natural
  numbers by automata: new proofs of {C}obham's and {S}emenov's theorems},
        date={1996},
     journal={Annals of Pure and Applied Logic},
      volume={77},
       pages={251\ndash 277},
}

\bib{ivp}{article}{
      author={Miller, Chris},
       title={Expansions of dense linear orders with the intermediate value
  property},
        date={2001},
     journal={The Journal of Symbolic Logic},
      volume={66},
       pages={1783\ndash 1790},
}

\bib{strong_dep}{unpublished}{
      author={Shelah, Saharon},
       title={Strongly dependent theories},
        note={arXiv:math/0504197v3},
}

\bib{simonnip}{misc}{
      author={Simon, Pierre},
       title={A guide to {NIP} theories},
        note={Availabe on the author's website--2015},
}

\bib{S}{article}{
      author={Simon, Pierre},
       title={On dp-minimal ordered structures},
        date={2011June},
     journal={The Journal of Symbolic Logic},
      volume={76},
      number={2},
       pages={448\ndash 460},
}

\bib{Simon_Walsberg}{article}{
      author={Simon, Pierre},
      author={Walsberg, Erik},
       title={Tame topology over dp-minimal structures},
        date={2015},
        note={arXiv:1509.08484v1},
}

\bib{Sz}{article}{
      author={Szemer\'edi, Endre},
       title={On sets of integers containing no $k$ elements in arithmetic
  progression},
        date={1975},
     journal={Acta Arithmetica},
      volume={27},
       pages={199\ndash 245},
}

\end{biblist}
\end{bibdiv}

\Addresses

\end{document}